\documentclass[12pt]{article}

\usepackage{amssymb}
\usepackage{amsmath}

\newtheorem{theorem}{Theorem}[section]
\newtheorem{lemma}[theorem]{Lemma}
\newtheorem{corollary}[theorem]{Corollary}
\newtheorem{proposition}[theorem]{Proposition}

\newtheorem{conjecture}[theorem]{Conjecture}

\newtheorem{question}[theorem]{Question}
\newenvironment{proof}{\normalsize {\sc Proof}.}{{\hfill $\Box$%
    \hskip - \parfillskip\bigskip}}

\newcommand{\Aut}{\mathop{\rm Aut}\nolimits}

\newcommand{\LL}{\mathop{\rm LL}\nolimits}

\newcommand{\Br}{\mathop{\rm Br}\nolimits}
\newcommand{\Rad}{\mathop{\rm Rad}\nolimits}

\def\bigcp{\mathop{\mathchoice 
    {\hbox{\sf\Large\lower 0.1\baselineskip\hbox{Y}}}%
    {\hbox{\sf\large\lower 0.1\baselineskip\hbox{Y}}}%
    {\hbox{\sf\normalsize\lower 0.1\baselineskip\hbox{Y}}}%
    {\hbox{\sf\tiny\lower 0.1\baselineskip\hbox{Y}}}%
}}
\def\bigtimes{\mathop{\mathchoice 
    {\hbox{\sf\Large\lower 0.1\baselineskip\hbox{X}}}%
    {\hbox{\sf\large\lower 0.1\baselineskip\hbox{X}}}%
    {\hbox{\sf\normalsize\lower 0.1\baselineskip\hbox{X}}}%
    {\hbox{\sf\tiny\lower 0.1\baselineskip\hbox{X}}}%
}}

\def\Sym(#1){\mathop{\rm Sym}(#1)}
\def\Sym(#1){S_{#1}}
\def\diag(#1){\mathop{\rm diag}(#1)}

\newenvironment{enumerate*}{%
     \begin{enumerate}%
     }%
    {\end{enumerate}}

\begin{document}

\title{Loewy lengths of blocks with abelian defect groups \footnote{This research was supported by the EPSRC (grant no. EP/M015548/1).} }

\author{Charles W. Eaton\footnote{School of Mathematics, University of Manchester, Manchester, M13 9PL, United Kingdom. Email: charles.eaton@manchester.ac.uk} and Michael Livesey\footnote{School of Mathematics, University of Manchester, Manchester, M13 9PL, United Kingdom. Email: michael.livesey@manchester.ac.uk}}

\date{29th July, 2016}
\maketitle


\begin{abstract}
We consider $p$-blocks with abelian defect groups and in the first part prove a relationship between its Loewy length and that for blocks of normal subgroups of index $p$. Using this,
we show that if $B$ is a $2$-block of a finite group with abelian defect group $D \cong C_{2^{a_1}} \times \cdots \times C_{2^{a_r}} \times (C_2)^s$, where $a_i > 1$ for all $i$ and $r \geq 0$, then $d < \LL(B) \leq 2^{a_1}+\cdots+2^{a_r}+2s-r+1$, where $|D|=2^d$. When $s=1$ the upper bound can be improved to $2^{a_1}+\cdots+2^{a_r}+2-r$. Together these give sharp upper bounds for every isomorphism type of $D$. A consequence is that when $D$ is an abelian $2$-group the Loewy length is bounded above by $|D|$ except when $D$ is a Klein-four group and $B$ is Morita equivalent to the principal block of $A_5$. We conjecture similar bounds for arbitrary primes and give evidence that it holds for principal $3$-blocks.
\end{abstract}


\section{Introduction}

Let $G$ be a finite group and $k$ be an algebraically closed field of characteristic $p$. For a block $B$ of $kG$, write $\LL (B)$ for the Loewy length of $B$, that is, the smallest $n \in \mathbb{N}$ such that $\Rad(B)^n=0$ and write $e_B$ for the block idempotent corresponding to $B$. We are interested in upper and lower bounds for $\LL(B)$ in terms of the isomorphism type of the defect groups. For $p$-solvable groups it is proved in~\cite{wi87} that the Loewy length is at most the order of a defect group and in~\cite{ko84} that it is strictly greater that $(p-1)d$, where $d$ is the defect of $B$. The upper bound does not hold when we remove the $p$-solvability hypothesis, as the principal $2$-block $B_0(kA_5)$ of $A_5$ (with Klein four defect groups) has Loewy length $5$, although it is tempting to think that blocks Morita equivalent to $B_0(kA_5)$ are the only counterexamples, as we will see is indeed the case for $p=2$. As remarked in~\cite{kks14} the lower bound $(p-1)d$ also does not hold when we remove the $p$-solvability hypothesis, although $d$ could still be a lower bound. In this paper we restrict our attention to blocks with abelian defect groups, but investigate bounds on the Loewy length of a block for arbitrary finite groups.

Using the results of~\cite{ekks14}, which depends on the classification of finite simple groups, we show that if $B$ is a $2$-block with abelian defect group $D \cong C_{2^{a_1}} \times \cdots \times C_{2^{a_r}} \times (C_2)^s$, where $a_i > 1$ for all $i$ and $r \geq 0$, then $d< \LL(B) \leq 2^{a_1}+\cdots+2^{a_r}+2s-r+1$, where $|D|=2^d$. When $s=1$ the upper bound can be improved to $2^{a_1}+\cdots+2^{a_r}+2-r$. Together these give sharp upper bounds for every isomorphism type of $D$: take $B$ to be the principal block of $k(C_{2^{a_1}} \times \cdots \times C_{2^{a_r}} \times SL_2(2^s))$ (see~\cite{alp79}). The lower bound was suggested in~\cite{kks14}, where there is an excellent discussion of lower bounds on Loewy lengths. Note that for blocks of $p$-solvable groups with abelian defect groups, we have $\LL(B) = \LL(kD)$.

A crucial element in establishing the above bounds is the consideration of the case where there is $N \lhd G$ such that $G=ND$. For $D$ elementary abelian we may apply the main result of~\cite{kk96}, which says that if $G$ is a split extension of $N$ (by a direct factor of $D$), then $B$ behaves as if $N$ were a direct factor of $G$. We generalize this result and use this to show how for arbitrary abelian $D$ we may compare the Loewy lengths of $B$ and the block of $N$ covered by $B$. This works for all primes, and we hope is of wider interest.

The paper is structured as follows. In Section \ref{indexp} we give the generalization of the theorem of~\cite{kk96} and the theorem concerning Loewy length of blocks when there is a normal subgroup of index $p$. In Section \ref{prelim} we give some preliminary results needed for the proof of the bounds, which we give in Section \ref{proof}. In Section \ref{otherprimes} we consider similar bounds for odd primes and make a conjecture.


\section{Normal subgroups of index $p$}
\label{indexp}

We first prove a generalisation of a theorem of Koshitani and K\"{u}lshammer~\cite{kk96}.

\begin{theorem}\label{thm:indexp}
Let $G$ be a finite group and $B$ a block of $kG$ with abelian defect group $D$ with primary decomposition $D=D_1\times\cdots\times D_t$. Let
$N\vartriangleleft G$ such that $G=ND_r$ for some $r\leq t$, $D_j\leq N$ for all $j\neq r$ and
$[G:N]=p$. Let $b$ be a block of $kN$ covered by $B$. Then $e_B=e_b$ and there exists an element $a\in Z(B)$ of
multiplicative order dividing $|D_r|$ such that $B=\bigoplus_{j=0}^{p-1}a^jkNe_b$.
\end{theorem}

\begin{proof}
We follow the proof of~\cite{kk96}, where it is assumed that $D_r\cap N=1$. Let $\mathcal{A}$ be a root of $B$ in
$C_G(D)$ and $N_G(D)_\mathcal{A}$ the stabilizer of $e_\mathcal{A}$ in $N_G(D)$. As in~\cite{kk96}, $e_B=e_b$, $Z(B)$ is a crossed
product of $Z(b)$ over $G/N$ and $N_G(D)_\mathcal{A}/C_G(D)$ is a $p'$-group.
\newline
\newline
Now $N_G(D)_\mathcal{A}/C_G(D)$ acts on both $D$ and $D\cap N$ by conjugation. Consider the
subgroup $Q\leq D$ generated by $D_r$ and all its $N_G(D)_\mathcal{A}$-conjugates. This group has exponent $|D_r|$ and so by~\cite[$\S5$, Theorem 2.2]{gor80}
$Q=Q_1\times Q_2$, where $Q_1$ and $Q_2$ are $N_G(D)_\mathcal{A}$-invariant, $Q_1$ is homocyclic of exponent $|D_r|$ and $Q_2$ has exponent strictly less than
$|D_r|$. As $[Q_1:Q_1\cap N]=p$ we have $Q_1\cap N\cong R_1\times R_2$, where $R_1$ is cyclic of order $|D_r|/p$ and $R_2$ is homocyclic of exponent $|D_r|$.
Again by~\cite[$\S5$, Theorem 2.2]{gor80} we can assume $R_1$ and $R_2$ are both $N_G(D)_\mathcal{A}$-invariant and therefore
by~\cite[$\S3$, Theorem 3.2]{gor80} $R_2$ must have an $N_G(D)_\mathcal{A}$-invariant complement $R$ in $Q_1$. So $R$ is cyclic, $N_G(D)_\mathcal{A}$-invariant,
$|R|=|D_r|$ and $[R:R\cap N]=p$.
\newline
\newline
Now $|\Aut(R)|_{p'}=p-1$ and the natural homomorphism $\Aut(R)\to\Aut(R/(R\cap N))$ is surjective and
so any non-trivial $p'$-element of $\Aut(R)$ acts non-trivially on $R/(R\cap N)$. However, $[R,N_G(D)_\mathcal{A}]\leq R\cap N$ and
so every element of $N_G(D)_\mathcal{A}$ acts trivially on $R/(R\cap N)$. Therefore, $N_G(D)_\mathcal{A}$ commutes with $R$.
\newline
\newline
Let $\alpha=\mathcal{A}^{C_G(C_D(N_G(D)_\mathcal{A}))}$. By Watanabe's result~\cite[Theorem 2(ii)]{wat89}, the map
\begin{align*}
f:Z(B)&\to Z(\alpha),\\
z&\mapsto\Br_{C_D(N_G(D)_\mathcal{A})}(z)e_\alpha
\end{align*}
is an isomorphism of $k$-algebras. Note that since $R$ commutes with $N_G(D)_\mathcal{A}$ we have that $R\leq Z(C_G(C_D(N_G(D)_\mathcal{A})))$ and also
$C_G(C_D(N_G(D)_\mathcal{A}))=RC_N(C_D(N_G(D)_\mathcal{A}))$. Therefore $\alpha$ and hence $Z(\alpha)$ are $G/N$-graded. Recall that $Z(B)$ is also $G/N$-graded and
since $e_\alpha\in kN$, $f$ respects these gradings. Setting $a:=f^{-1}(e_\alpha g)$, where $R=\langle g\rangle$, proves the theorem.

\end{proof}

\begin{theorem}\label{thm:redp}
Assume we are in the setting of Theorem~\ref{thm:indexp}. Then

(i) $\Rad(B)=\Rad(b)Z(B)+(1-a)B$;

(ii) $\LL(b) < \LL(B)\leq \LL(b)+|D_r|-|D_r|/p$.
\end{theorem}

\begin{proof}
(i) Certainly $\Rad(b)Z(B)$ and $(1-a)B$ are nilpotent ideals of $B$ and so
\begin{align*}
\Rad(B)\supseteq\Rad(b)Z(B)+(1-a)B.
\end{align*}
Then
\begin{align*}
B/(\Rad(b)Z(B))=B/(\Rad(b)B)\cong(b/\Rad(b))\otimes_{b}B\cong(b/\Rad(b))\otimes_kkC_p,
\end{align*}
where the last isomorphism is given by $(x\otimes a)\mapsto(x\otimes g)$, where $C_p=\langle g\rangle$. Therefore
\begin{align*}
\Rad(B)/(\Rad(b)Z(B)+(1-a)B)
\end{align*}
is semisimple and so $\Rad(B)=\Rad(b)Z(B)+(1-a)B$.

(ii) If $\gamma_1,\dots,\gamma_s \in \Rad(b)$ with $\gamma_1\dots\gamma_s\neq0$, then $\gamma_1\dots\gamma_s (1-a)$ is a non-zero product of elements of $\Rad(B)$, and so $\LL(B)>\LL(b)$.

Next suppose $\beta_1,\ldots,\beta_s \in \Rad(B)$ with $\beta_1\dots\beta_s\neq0$ with $\beta_i=\Rad(B)$. Then by (i) there exist $\gamma_1,\dots,\gamma_s$ such that
$\gamma_1\dots\gamma_s\neq0$ and each $\gamma_i$ is either $(1-a)$ or is an element of $\Rad(b)$. Say $t$ of the $\gamma_i$'s are
$(1-a)$'s. Then clearly $t<|D_r|$ and since $(1-a)^p\in\Rad(b)$ we also have that $\lfloor t/p\rfloor+s-t<\LL(b)$.
As increasing $t$ can only decrease the left hand side of this inequality, it is also true that $\lfloor(|D_r|-1)/p\rfloor+s-(|D_r|-1)<\LL(b)$
and so $|D_r|/p-1+s-(|D_r|-1)<\LL(b)$ giving that $s<\LL(b)+|D_r|-|D_r|/p$ and the theorem is proved.
\end{proof}


\section{Further preliminary results}
\label{prelim}

The following result is well-known, but we are not aware of an explicit reference.

\begin{lemma}
\label{DconjugatesgenerateG}
Let $G$ be a finite group and $B$ be a $p$-block of $G$ with defect group $D$. Let $N \lhd G$ and let $b$ be a block of $N$ covered by $B$. If $D \leq N$, then $\LL(B)=\LL(b)$.
\end{lemma}

\begin{proof}
The proof given in~\cite[Lemma 4.1]{km01} for the case $G/N$ is a $p'$-group carries through in the general situation, but we include a simple argument suggested by Markus Linckelmann.

Let $I_G(b)$ be the stabilizer of $b$ in $G$. Since there is a block of $I_G(b)$ covering $b$ and Morita equivalent to $B$, we may assume that $G=I_G(b)$. Then $\Rad(b)B \subseteq \Rad(B)$. Since $D \leq N$, every $B$-module is relatively $N$-projective, and so in particular $B/\Rad(b)B$ and every quotient module is. Now the restriction of $B/\Rad(b)B$ to $kN$ is a semisimple $b$-module, so it follows that $B/\Rad(b)B$ is semisimple as a $B$-module. Hence $\Rad(B) \subseteq \Rad(b)B$, and so $\Rad(B)=\Rad(b)B$. Again since $b$ is $G$-stable, we have $\LL(B)=\LL(b)$.
\end{proof}

\begin{lemma}
\label{product}
Let $G_1, \ldots, G_t$ be finite groups and $G=G_1 \times \cdots \times G_t$. For each $i$ let $B_i$ be a block of $G_i$ and let $B=B_1 \otimes \cdots \otimes B_s$ be the block of $G$ covering each $B_i$. Let $D_i$ be a defect group of $B_i$, so $D=D_1 \times \cdots \times D_t$ is a defect group of $B$. Then $\LL(B)=(\sum_{i=1}^t \LL(B_i))-t+1$.

In particular, suppose $|D|=p^d$, $|D_i|=p^{d_i}$ and $D_i \cong C_{p^{a_{i,1}}} \times \cdots \times C_{p^{a_{i,r_i}}} \times (C_p)^{s_i}$, where each $a_{i,j}>1$. If $d_i<\LL(B_i) \leq p^{a_{i,1}}+\cdots+p^{a_{i,r_i}}+ps_i-r_i+1$ for each $i$, then
$d<\LL(B) \leq \sum_{i=1}^t \sum_{j=1}^{r_i}p^{a_{i,j}}+p\sum_{i=1}^t s_i- (\sum_{i=1}^tr_i)+1$. If, in addition, $\LL(B_i) \leq p^{a_{i,1}}+\cdots+p^{a_{i,r_i}}+ps_i-r_i$ for at least one $i$, then $\LL(B) \leq \sum_{i=1}^t \sum_{j=1}^{r_i}p^{a_{i,j}}+p\sum_{i=1}^t s_i- (\sum_{i=1}^tr_i)$.
\end{lemma}

\begin{proof}
The first part is~\cite[1.1]{ajl83} and the second part is just a direct application of the first.
\end{proof}

Before proceeding we recall the definition and some properties of the generalized Fitting subgroup $F^*(G)$ of a finite group $G$. Details may be found in~\cite{asc00}.

A \emph{component} of $G$ is a subnormal quasisimple subgroup of $G$. The components of $G$ commute, and we define the \emph{layer} $E(G)$ of $G$ to be the normal subgroup of $G$ generated by the components. It is a central product of the components. The \emph{Fitting subgroup} $F(G)$ is the largest nilpotent normal subgroup of $G$, and this is the direct product of $O_r(G)$ for all primes $r$ dividing $|G|$. The \emph{generalized Fitting subgroup} $F^*(G)$ is $E(G)F(G)$. A crucial property of $F^*(G)$ is that $C_G(F^*(G)) \leq F^*(G)$, so in particular $G/F^*(G)$ may be viewed as a subgroup of $\Aut(F^*(G))$.

\begin{lemma}
\label{reduce}
Let $B$ be a block of a finite group $G$ with abelian defect group $D$. Then there is a group $H$ and a block $C$ of $H$ with defect group $D_C \cong D$ such that $\LL(C)=\LL(B)$ and the following are satisfied:

(i) $C$ is quasiprimitive, that is, for every normal subgroup $N$ of $D$, every block of $N$ covered by $C$ is $H$-stable;

(ii) $H$ is generated by the defect groups of $C$;

(iii) If $N \lhd H$ and $C$ covers a nilpotent block of $N$, then $N \leq Z(H)O_p(H)$. In particular, $O_{p'}(H) \leq Z(H)$;

(iv) $O_p(H) \leq Z(H)$;

(v) Every component of $H$ is normal in $H$;

(vi) If $B$ is the principal block of $G$, then $C$ is the principal block of $H$.

\end{lemma}

\begin{proof}
Consider pairs $([G:O_{p'}(Z(G))],|G|)$ with the lexigraphic ordering. There are three processes, labelled (a), (b), (c), which will be applied repeatedly and in various combinations. We describe these processes and show that they strictly reduce $([G:O_{p'}(Z(G))],|G|)$ when applied non-trivially, so that repeated application of (a), (b) and (c) terminates.

(a) Let $N \lhd G$ and let $b$ be a block of $N$ covered by $B$. Write $I=I_G(b)$ for the stabilizer of $b$ in $G$, and $B_I$ for the Fong-Reynolds correspondent. Now $B_I$ is Morita equivalent to $B$ and they have isomorphic defect groups. Clearly $O_{p'}(Z(G)) \leq O_{p'}(Z(I))$, and if $I \neq G$, then $[I:O_{p'}(Z(I))]<[G:O_{p'}(Z(G))]$. Process (a) involves replacing $B$ by $B_I$.

(b) is the replacement of $G$ by $N=\langle D^g:g \in G \rangle$ and $B$ by any block $b$ of $N$ covered by $B$, as in Lemma \ref{DconjugatesgenerateG}. In this case $[N:O_{p'}(Z(N))] \leq [G:O_{p'}(Z(G))]$ and $|N| < |G|$ if $N \neq G$. Since defect groups of $b$ are intersections of defect groups of $B$ with $N$ (see~\cite[15.1]{alp86}), it follows that $B$ and $b$ share a defect group.

(c) Let $N \lhd G$ and suppose that $B$ covers a nilpotent block $N$ of $G$ such that $N \not\leq Z(G)O_p(G)$. Let $b'$ be a block of $Z(G)N$ covered by $B$ and covering $b$. By performing (a) first we may assume that $b'$ is $G$-stable. Further $b'$ must also be nilpotent. Using the results of~\cite{kp90}, as outlined in~\cite[Proposition 2.2]{ekks14}, $B$ is Morita equivalent to a block $\tilde{B}$ of a central extension $\tilde{L}$ of a finite group $L$ by a $p'$-group such that there is an $M \lhd L$ with $M \cong D \cap (Z(G)N)$, $G/Z(G)N \cong L/M$, and $\tilde{B}$ has defect group isomorphic to $D$. Note that $[\tilde{L}:O_{p'}(Z(\tilde{L}))] \leq |L| = [G:Z(G)N]|D \cap (Z(G)N)| < [G:O_{p'}(Z(G))]$ and that $M \leq O_p(\tilde{L})Z(\tilde{L})$. Process (c) consists of replacing $G$ by $\tilde{L}$ and $B$ by $\tilde{B}$.

Repeated application of (a), (b) and (c) must eventually terminate, in which case we are left with $H$ and a block $C$ of $H$ with defect group $D_C \cong D$ satisfying conditions (i)-(iii).

To see that $H$ and $C$ satisfy (iv), note that $D \leq C_G(O_p(G)) \lhd G$, so (iv) is a consequence of (ii).

(v) Write $L_1,\ldots,L_m$ for the components of $G$. We may assume that $B$ is quasiprimitive, $G$ is generated by the defect groups of $B$ and that $B$ does not cover any nilpotent block with non-central defect groups. As above the generalised Fitting subgroup $F^*(G)=E(G)F(G)$, where $E(G)=L_1\cdots L_m \lhd G$. Since by (iv) we may assume $O_p(G)\leq Z(G)$ and by (iii) we may assume $O_r(G)\leq Z(G)$ for primes $r\neq p$, we assume that $F^*(G)=E(G)Z(G)$. By a similar argument we also have that $Z(F^*(G))=Z(G)$. Let $\varphi : G \rightarrow S_m$ be the homomorphism given by the permutation action on the components. Now $D \cap F^*(G)$ is a defect group for a (non-nilpotent) block of $F^*(G)$ covered by $B$. Hence $(D \cap F^*(G))/O_p(Z(G))$ is a defect group for a block of $F^*(G)/O_p(Z(G))$. Therefore $(D \cap F^*(G))/O_p(Z(G))$ is a radical $p$-subgroup of $F^*(G)/O_p(Z(G))$ (recall that a $p$-subgroup $Q$ of a finite group of a finite group $H$ is radical if $Q=O_p(N_G(Q))$ and that defect groups are radical $p$-subgroups) and so $(D \cap F^*(G))Z(G)/Z(G)$ is a radical $p$-subgroup of $F^*(G)/Z(G)\cong (L_1Z(G)/Z(G)) \times \cdots \times (L_mZ(G)/Z(G))$. Note that $(D \cap F^*(G))Z(G)/Z(G)$ is not necessarily a defect group, hence our move to the weaker condition of being a radical $p$-subgroup. By~\cite[Lemma 2.2]{ou95} we have therefore $(D \cap F^*(G))Z(G)/Z(G) = D_1 \times \cdots \times D_m$, where $D_i = (D \cap F^*(G))Z(G)/Z(G)) \cap (L_iZ(G)/Z(G))$. Since $B$ is quasiprimitive it follows that each $D_i$ is non-trivial (otherwise $B$ covers a nilpotent block of a non-central normal subgroup, namely the subgroup generated by $Z(G)$ and the orbit of $L_i$). Hence since $D$ is abelian we have $D \leq ker(\varphi)$, and so (v) follows since $G$ is generated by the conjugates of $D$.

To prove (vi), it suffices to show that the processes (a), (b) and (c) respect principal blocks. If $N \lhd G$ and $B$ is principal, then $B$ covers the principal block of $N$. Also, the Fong-Reynolds correspondence takes principal blocks to principal blocks. It follows that if $B$ is the principal block of $G$, then the block constructed in (a) and (b) may be taken to be the principal block. In (c), if $B$ is the principal block and $B$ covers a nilpotent block $b$ of $N \lhd G$, then $b$ is the principal block of $N$ and by a theorem of Frobenius (see~\cite[7.4.5]{gor80}) we have that $N$ has a normal $p$-complement. We may then apply the reduction to $O_{p'}(N) \lhd G$. But $B$ covers the principal block of $O_{p'}(N)$, i.e., $O_{p'}(N)$ lies in the kernel of $B$, and we may replace $B$ by the principal block of $G/O_{p'}(N)$.

\end{proof}

\begin{proposition}~\cite[Theorem 6.1]{ekks14}
\label{classification}

Let $k$ be an algebraically closed field of characteristic $2$, $G$ a quasi-simple group and $B$ be a block of $kG$ with abelian defect group $P$, then one (or more) of the following holds:

(i) $G/Z(G)$ is one of $A_1(2^a)$, $\,^2G_2(q)$ (where $a>1$ and $q \geq 27$ is a power of $3$ with odd exponent), or $J_1$, $B$ is the principal block and $P$ is elementary abelian.

(ii) $G$ is $Co_3 $, $B$ is a non-principal block, $P \cong C_2 \times C_2 \times C_2 $.

(iii) There exists a finite group $\tilde G$ such that $G \unlhd \tilde G $, $Z(G) \leq Z(\tilde G) $ and such that $B$ is covered by a nilpotent block of $\tilde G$.

(iv) $B$ is Morita equivalent to a block $C$ of $kL$ where $L=L_0 \times L_1$  is  a subgroup of $G$ such that the following holds: The defect groups of $C$ are isomorphic to $P$, $L_0$ is abelian and the block of $kL_1 $ covered by $C$ has Klein four defect groups.

\end{proposition}


\section{Proof of the main result}
\label{proof}

\begin{theorem}
Let $B$ be a $2$-block with abelian defect group $D$. Suppose $D \cong C_{2^{a_1}} \times \cdots \times C_{2^{a_r}} \times (C_2)^s$, where $a_i > 1$ for all $i$ and $r \geq 0$. Write $|D|=2^d$.

Then $d < \LL(B) \leq 2^{a_1}+\cdots+2^{a_r}+2s-r+1$. If $s=1$, then $\LL(B) \leq 2^{a_1}+\cdots+2^{a_r}+2-r$.
\end{theorem}

\begin{proof}
%
We may assume that $B$ satisfies conditions (i) to (v) of Lemma~\ref{reduce}.

Next suppose that $N$ is a normal subgroup of $G$ of index $2$ and $b$ is a block of $N$ covered by $B$. Since $B$ is quasiprimitive $G$ stabilizes $b$, so by~\cite[$\S15$, Theorem 1 (4)]{alp86} there exists a block $B'$ of $G$ covering $b$ with defect group $D'$ such that $D'N=G$. Now by Theorem~\ref{thm:indexp}, $B'$ and $b$ must share an idempotent and so $B'$ must equal $B$. We can now apply Theorem~\ref{thm:redp}(ii), so that it suffices to prove the Theorem for $b$. By repeated application of this, and possibly further application of Lemma~\ref{reduce}, we may assume that $G=O^2(G)$.

By the arguments in Lemma~\ref{reduce}(v) and the fact that generalised Fitting subgroups are self-centralising we may assume that $$S_1 \times \cdots \times S_m \leq G/Z(G) \leq \Aut(S_1) \times \cdots \Aut(S_m) \leq \Aut(S_1 \times \cdots \times S_m)$$ for normal non-abelian simple groups $S_1,\ldots ,S_m$. Continuing with the notation from Lemma~\ref{reduce}(v) we have $L_iZ(G)/Z(G) \cong S_i$ and $L_i \lhd G$ for each $i$. By the Schreier conjecture $G/E(G)$ is solvable, and hence, as we are assuming both $G=O^2(G)$ and $G=O^{2'}(G)$, we may assume that $G=E(G)$.

The remainder of the proof proceeds almost as in that of~\cite[Theorem 8.3]{ekks14}. There it is proved that we may assume (by replacing $B$ by a Morita equivalent block of another finite group if necessary) that $G=E/Z$, where $E=U \times V \times W$ with $Z \leq O_2(Z(E)) \cap U$, $B$ is a product of blocks $B_{U/Z}$, $B_V$ and $B_W$ of $U/Z$, $V$ and $W$ respectively and $D=D_{U/Z} \times D_V \times D_W$ with the following properties: $B_{U/Z}$ is a block of $U/Z$ with defect group $D_{U/Z}$ and is Morita equivalent to its Brauer correspondent in $N_{U/Z}(D_{U/Z})$; $V=V_1 \times \cdots \times V_t$ and $B_V$ is a block of $V$ with defect group $D_V$ and which is a product of blocks $B_{V_i}$ of $V_i$ with Klein four defect groups; $W=W_1 \times \cdots \times W_u$ and $B_W$ is a block of $W$ with defect group $D_W$ and which is a product of blocks $B_{W_i}$ of $W_i$ satisfying condition (i) or (ii) of Proposition \ref{classification}.

By Lemma~\ref{DconjugatesgenerateG} $\LL(B_{U/Z}) = \LL(kD_{U/Z})$, and so $B_{U/Z}$ satisfies the required inequality. Since a block with Klein four defect groups has Loewy length $3$ or $5$, by Lemma \ref{product} $B_V$ also satisfies the inequality. It remains to consider the blocks satisfying (i) or (ii) of Proposition \ref{classification}. By~\cite{alp79} the Loewy length of the principal block of $SL_2(2^a)$ is $2a+1$. By~\cite{ok97}, for all $m$ the principal $2$-block of $^2G_2(3^{2m+1})$ is Morita equivalent to that of $\Aut(SL_2(8))$. Since $[\Aut(SL_2(8)):SL_2(8)]=3$, by Lemma \ref{DconjugatesgenerateG} the Loewy length in this case is that of the principal block of $SL_2(8)$, i.e., $7$. By~\cite{go97} and~\cite{lm80} we have $\LL(B_0(kJ_1))=7$. By~\cite[1.5]{kmn11} the principal block of $Co_3$ is Morita equivalent to that of $\Aut(SL_2(8))$, so again the Loewy length is $7$.

Next we note that if $s=1$ then $V=\{1\}$ and the relevant block ($B_{U/Z}$ or $B_W$, depending on whether $D_{U/Z}$ or $D_W$ has a factor isomorphic to $C_2$) satisfies the stronger upper bound.

The result then follows by Lemma \ref{product}.
\end{proof}

\begin{corollary}
Let $B$ be a $2$-block with abelian defect group $D$. Then $\LL(B) \leq |D|$ unless $D$ is a Klein four group.
\end{corollary}


\section{Other primes}
\label{otherprimes}

There is relatively little evidence for similar bounds for odd primes, but it is tempting to conjecture the following:

\begin{conjecture}
\label{oddprimes}
Let $B$ be a block of a finite group $G$ with abelian defect group $D \cong C_{p^{a_1}} \times \cdots \times C_{p^{a_r}} \times (C_p)^s$, where $a_i > 1$ for all $i$ and $r \geq 0$. Write $|D|=p^d$. Then $$d<\LL(B) \leq p^{a_1}+\cdots+p^{a_r}+ps-r+ \delta,$$ where $\delta=1$ if $s$ is even and $\delta=0$ if $s$ is odd.
\end{conjecture}

Note that when $D$ is cyclic $\LL(B)$ may be computed from the Brauer tree, and is bounded above by $\left( \frac{|D|-1}{e(B)} \right) e(B) +1 = |D|$, where $e(B)$ is the number of simple $B$-modules (here $\frac{|D|-1}{e(B)}$ is the multiplicity of an exceptional vertex and $e(B)$ is an upper bound for the number of edges emanating from a vertex). By~\cite[2.8]{kks14} $\LL(B) \geq \frac{|D|-1}{e(B)} + 1 > d$ since $e(B) \leq p-1$.

We discuss the case $p=3$ below, but mention now that the conjectured bound is achieved in this case by taking the principal block of $C_{3^{a_1}} \times \cdots \times C_{3^{a_r}} \times (M_{23})^{s/2}$ if $s$ is even or $C_{3^{a_1}} \times \cdots \times C_{3^{a_r}} \times (M_{23})^{(s-1)/2} \times C_3$ if $s$ is odd, where $M_{23}$ is the sporadic simple group of that name.

A key special case of Conjecture \ref{oddprimes} is where $D \cong C_p \times C_p$. There is a relative scarcity of computed examples when $p>3$, and we do not know of any examples in this case which make the upper bound sharp for $s>1$. Hence we ask the following:

\begin{question} Is there always a block $B$ of some finite group $G$ with defect group $C_p \times C_p$ and $\LL(B) = 2p+1$?
\end{question}

In~\cite{ko03} Koshitani describes the finite groups whose Sylow $3$-subgroups are abelian, and in~\cite{km01} it is shown that any principal $3$-block with abelian defect groups of order $3^2$ has Loewy length $5$ or $7$. Based on this we show the following, unfortunately involving further unpublished preprints. Note that at present it is not realistic to calculate the Loewy length for the principal block of the sporadic simple group $O'N$ (with elementary abelian Sylow $3$-subgroups of order $3^4$) by computer or otherwise. Conjecture \ref{oddprimes} predicts that the Loewy length lies between $5$ and $13$.

\begin{proposition}
Suppose that the Loewy length of the principal $3$-block of $O'N$ is at least $5$ and at most $13$. Let $G$ be a finite group with Sylow $3$-subgroup $D\cong C_{3^{a_1}} \times \cdots \times C_{3^{a_r}} \times (C_3)^s$, where $a_i > 1$ for all $i$ and $r \geq 0$, and let $B$ be the principal $3$-block of $G$. Write $d$ for the defect of $B$. Then $$d<\LL(B) \leq 3^{a_1}+\cdots+3^{a_r}+3s-r+1.$$
\end{proposition}

\begin{proof}
We may assume that $B$ satisfies conditions (i) to (vi) of Lemma \ref{reduce}. Since $O_{3'}(G)$ is in the kernel of the principal block, we may further assume that $O_{3'}(G)=1$. By~\cite{ko03}, which is based on~\cite{fo96},
$$G=O_3(G) \times G_1 \times \cdots \times G_n$$ where $n \geq 0$ and each $G_i$ is either a non-abelian simple group with cyclic Sylow $3$-subgroups or is one of the following:

(i) $A_6$, $A_7$, $A_8$, $M_{22}$, $M_{23}$, $HS$, $O'N$;

(ii) $PSL_3(q)$ where $3|q-1$ but $3^2 \not| q-1$;

(iii) $PSU_3(q^2)$ where $3|q+1$ but $3^2 \not| q+1$;

(iv) $PSp_4(q)$ where $3|q-1$;

(v) $PSp_4(q)$ where $q>2$ and $3|q+1$;

(vi) $PSL_4(q)$ were $q>2$ and $3|q+1$;

(vii) $PSU_4(q^2)$ where $3|q-1$;

(viii) $PSL_5(q)$ where $3|q+1$;

(ix) $PSU_5(q^2)$ where $3|q-1$;

(x) $PSL_2(3^m)$ where $m \geq 2$.

By Lemma \ref{product} it suffices to check the inequalities for each factor group in turn. We have treated the case that $D$ is cyclic above, so it suffices to consider each of the cases (i)-(x) above in turn.

 (i)-(iii) In all cases except $O'N$ we have $D \cong C_3 \times C_3$ and by~\cite{km01} $\LL(B) = 5$ or $7$.

 (iv), (vii), (ix) In these cases $D \cong C_{3^a} \times C_{3^a}$ for some $a \geq 1$. It follows from~\cite[3.6]{pu90} that $B$ is Morita equivalent to the principal block of $N_G(D)$, and so by Lemma \ref{DconjugatesgenerateG} $\LL(B)=\LL(kD) = 2 \cdot 3^a - 1$.

 (v)  In this case also $D \cong C_{3^a} \times C_{3^a}$ for some $a \geq 1$. By~\cite{ko11} $B$ is Morita equivalent to the principal block of $N_G(Q)$ where $D=Q \times R$ and $|Q|=3^a$, and the result is true by induction.

 (vi), (viii) In these cases $\LL(B)$ is equal to the Loewy length of the principal block of the corresponding general linear group, and the weight of this block is two. By~\cite[Theorem 1]{tu02} this is Morita equivalent to the principal block of $GL_2(q) \wr S_2$. By Lemma \ref{DconjugatesgenerateG} and Lemma \ref{product} $\LL(B_0(k(GL_2(q) \wr S_2))) = 2 \LL(B_0(GL_2(q)) - 1$. Since $GL_2(q)$ has a cyclic Sylow $3$-subgroup we are done in this case.

 (x) By~\cite{ajl83} $\LL(B_0(PSL_2(3^m))) = 2m+1$ and we are done.

\end{proof}

Finally we prove the result concerning blocks of $p$-solvable groups with abelian defect groups mentioned in the introduction, which in particular implies that such blocks satisfy sharper bounds than in general.

\begin{proposition}
Let $B$ be a block of a $p$-solvable group $G$ with abelian defect group $D$. Then $\LL(B)=\LL(kD)$.
\end{proposition}

\begin{proof}
By Lemma \ref{reduce} we may assume that $O_{p'}(G) \leq Z(G)$ and $C_G(O_p(G)) =G$, so by~\cite[6.3.2]{gor80} $G=C_G(O_p(G))  \leq O_p(G)Z(G)$ and the result follows.
\end{proof}


\end{document}